\documentclass[a4paper,12pt]{amsart}
\usepackage{amsmath}
\usepackage{amsthm}
\usepackage{amscd}
\usepackage{amssymb}
\usepackage{amsfonts}
\usepackage{latexsym}
\usepackage[mathscr]{eucal}

\newtheorem{thm}{Theorem}[section]
\newtheorem{prop}[thm]{Proposition}
\newtheorem{lemma}[thm]{Lemma}
\newtheorem{cor}[thm]{Corollary}

\def\graph{\mathop{\rm {graph}}\nolimits}

\def\Im{\mathop{\rm {Im}}\nolimits}
\def\codim{\mathop{\rm {codim}}\nolimits}

\begin{document}
\title{Relative position of three subspaces
in a Hilbert space}
\author{Masatoshi Enomoto}
\address[Masatoshi Enomoto]{
1-6-13, Nogami, Takarazuka,Hyogo, Japan} 
\email{enomotoma@hotmail.co.jp}     

\author{Yasuo Watatani}
\address[Yasuo Watatani]{Department of Mathematical Sciences, 
Kyushu University, Hakozaki, 
Fukuoka, 812-8581,  Japan}
\email{watatani@math.kyushu-u.ac.jp}
\maketitle
\begin{abstract}
We study the relative position of
three subspaces in a separable 
infinite-dimensional Hilbert space. 
In the finite-dimensional case, Brenner 
described the general position of three 
subspaces completely. 
We extend it to a certain class of three subspaces in 
an infinite-dimensional Hilbert space. 
We also give a partial result which gives a condition 
on a system to have a (dense) decomposition containing a pentagon.

 
\medskip\par\noindent
AMS SUBJECT CLASSIFICATION: 46C07, 47A15, 15A21, 16G20, 16G60.

\end{abstract}

\section{Introduction}
We study the relative position of
three subspaces in a separable 
infinite-dimensional Hilbert space. 

The relative position of one
subspace of a Hilbert space is extremely simple and determined 
by the dimension and the co-dimension of the subspace.
It is a well known fact that the relative position of  two subspaces
$E$ and $F$ in a Hilbert space $H$  can be described completely up to
unitary equivalence  as in Araki \cite{Ar},Davis \cite{Da}, 
Dixmier \cite{Di} and 
Halmos \cite{Ha1}.  The Hilbert
space is the direct sum of five subspaces:
$$
H = (E \cap F) \oplus (\text{the rest}) \oplus (E \cap F^{\perp})
\oplus (E^{\perp} \cap F) \oplus (E^{\perp} \cap F^{\perp}).
$$
In the rest part,  $E$ and $F$ are in generic position and the
relative position is described only by \lq\lq the angles'' between them.

We disregard \lq\lq the angles'' and study the still-remaining fundamental 
feature of the relative position of subspaces. This is the reason why  we 
use  bounded invertible operators instead of unitaries to define 
isomorphisms in our paper. 

Let $H$ be a Hilbert space and $E_1, \dots E_n$ be $n$ subspaces 
in $H$.  Then we say that  ${\mathcal S} = (H;E_1, \dots , E_n)$  
is a system of $n$ subspaces in $H$ or an $n$-subspace system in $H$.
Let ${\mathcal T} = (K;F_1, \dots , F_n)$  
be  another system of $n$-subspaces in a Hilbert space $K$. 
We say that systems ${\mathcal S}$ and ${\mathcal T}$ are 
{\it isomorphic} if there is a  bounded invertible 
operator $\varphi : H \rightarrow K$ satisfying that  
$\varphi(E_i) = F_i$ for $i = 1,\dots ,n$. See also 
Sunder \cite{S} for other topics on $n$-subspaces. 

In \cite{B}  S. Brenner gave a complete description of systems of 
three subspaces up to isomorphims when an ambient space $H$ is 
finite-dimensional.  

A system ${\mathcal S}$ is called indecomposable if ${\mathcal S}$ 
can not be decomposed into a nontrivial direct sum. If the 
ambient Hilbert space $H$ is finite-dimensional, then any system of 
$n$ subspaces in $H$ is a finite direct sum of indecomposable 
systems.

Let $\mathcal{S}=(H;E_{1},E_{2},E_{3})$
be an indecomposable system of three subspaces 
in a finite-dimensional Hilbert space $H$.
Then  $\mathcal S$  is isomorphic to 
one of the following eight trivial  systems 
${\mathcal S}_1, \dots, {\mathcal S}_8$  
 and one non-trivial  system ${\mathcal S}_9$:
\[
{\mathcal S}_1 = (\mathbb C;0,0,0),
\ \ {\mathcal S}_2 = (\mathbb C;\mathbb C,0,0),
\ \ {\mathcal S}_3 = (\mathbb C; 0,\mathbb C,0),
\]
\[
{\mathcal S}_4 = (\mathbb C; 0,0, \mathbb C), 
\ \ {\mathcal S}_5 = (\mathbb C;\mathbb C,\mathbb C,0),
\ \ {\mathcal S}_6 = (\mathbb C;\mathbb C,0,\mathbb C),
\]
\[
{\mathcal S}_7 = (\mathbb C; 0,\mathbb C,\mathbb C),
\ \ {\mathcal S}_8 = (\mathbb C; \mathbb C , \mathbb C,\mathbb C), 
\ \ {\mathcal S}_9 = (\mathbb C^2; \mathbb C (1,0), 
\mathbb C (0,1),\mathbb C (1,1)).
\]

See, for example,   \cite{GP}, \cite{HLR} or \cite{EW1} on indecomposable 
systems of $n$ subspaces. 

Therefore we have the following theorem of Brenner: 

\begin{thm}[Brenner \cite{B}]
Let $\mathcal{S}=(H;E_{1},E_{2},E_{3})$
be a system of three subspaces 
in a finite-dimensional Hilbert space $H$.
Then  $\mathcal S$  is isomorphic to the following 
${\mathcal T} = (H;F_1,F_2,F_3 )$ such that there exist 
subspaces $S, N_1,N_2, N_3, M_1, M_2, M_3, Q, L$ of $H$ satisfying 
that $Q$  has a form
$$
(Q; Q_1,Q_2, Q_3) := 
(K \oplus K ; K \oplus 0, 0 \oplus K, \{(x,x) \ | x \in K\})
$$
of double triangle and 
\begin{align*}
  H & = S \oplus N_1 \oplus N_2 \oplus N_3 \oplus 
      M_1 \oplus M_2 \oplus M_3 \oplus Q \oplus L \\
F_1 & = S \oplus 0_{ \ \ } \oplus N_2 \oplus N_3 \oplus 
      M_1 \oplus 0_{ \ \ } \oplus 0_{ \ \ } \oplus Q_1 \oplus 0_{ \ \ }\\  
F_2 & = S \oplus N_1 \oplus 0_{ \ \ } \oplus N_3 \oplus 
      0_{ \ \ } \oplus M_2 \oplus 0_{ \ \ } \oplus Q_2 \oplus 0_{ \ \ } \\  
F_3 & =  S \oplus N_1 \oplus N_2 \oplus 0_{ \ \ } \oplus 
     0_{ \ \ } \oplus 0_{ \ \ } \oplus M_3 \oplus Q_3 \oplus 0_{ \ \ } \\  
\end{align*} 

\end{thm}

\noindent
{\bf Remark.}  In the above decomposition, we can choose 
${\mathcal T}$ such that 
 $S = F_1 \cap F_2 \cap F_3$,  
     $N_1 = F_1^{\perp} \cap F_2 \cap F_3$,
     $N_2 = F_1 \cap F_2^{\perp} \cap F_3$,
     $N_3 = F_1 \cap F_2 \cap F_3^{\perp}$,
     $M_1 = F_1 \cap F_2^{\perp} \cap F_3^{\perp}$
     $M_2 = F_1^{\perp} \cap F_2 \cap F_3^{\perp}$,
     $M_3 = F_1^{\perp} \cap F_2^{\perp} \cap F_3$ 
and  $L = F_1^{\perp} \cap F_2^{\perp} \cap F_3^{\perp}$ . 
But we should be careful that the isomorphism by an 
invertible operator does not 
preserve the orthogonality.

The aim of our paper is to extend the Brenner's theorem 
to a certain class of three subspaces in 
an infinite-dimensional Hilbert space.

The above Brenner's theorem says that any system of three 
subspaces of a finite-dimensional Hilbert space is decomposed as a
direct sum of a distributive part (or Boolean part) 
$$
S \oplus N_1 \oplus N_2 \oplus N_3 \oplus 
      M_1 \oplus M_2 \oplus M_3 \oplus L 
$$
and a 
non-distributive part $Q$. 
Furthermore the non-distributive part 
$Q = K \oplus K$ has a  typical form
$$
(K \oplus K ; K \oplus 0, 0 \oplus K, \{(x,x) \ | x \in K\})
$$
of double triangle. The double triangle is the {\it only} 
obstruction of distributive law in finite-dimensional case. 
We study this type of decomposition for a certain class of 
systems of three subspaces for an infinite-dimensional 
Hilbert space. In order to proceed this type of  decomposition, 
we should recall the following basic facts on the subspace lattice 
structure: In general, a lattice is distributive if and only if it has 
neither a double triangle nor a pentagon as a sublattice, 
see \cite{Gr} for example.  In 
the subspace lattices of an infinite dimensional Hilbert space, 
there occur both double triangles and pentagons. 
A von Neumann algebra $M$ 
is commutative  if and only if the lattice of the projections in $M$ 
is distributive.  
A von Neumann algebra $M$ 
is finite if and only if the lattice of the projections in $M$ 
has no pentagons if and only if the lattice of the projections in $M$ 
is modular. Therefore we understand that the general case is far 
beyond having a Brenner type decomposition. 
 
For any bounded linear operator $A$ on a Hilbert space $K$, we can 
associate a system ${\mathcal S}_A$ of four subspaces in 
$H = K \oplus K$ by 
\[
{\mathcal S}_A = (H;K\oplus 0,0\oplus K,\graph A, \{(x,x) ; x \in K\}).
\]
Two such systems ${\mathcal S}_A$ and ${\mathcal S}_B$  are isomorphic 
if and only if the two operators $A$ and $B$ are similar. The 
direct sum of such  systems corresponds to the direct sum of 
the operators. In this sense the theory of operators is 
included into the theory of relative positions of four subspaces. 
In particular on a finite dimensional space, Jordan blocks correspond 
to indecomposable systems. Moreover on an infinite dimensional 
Hilbert space, the above system ${\mathcal S}_A$ is indecomposable 
if and only if $A$ is strongly irreducible, which is an 
infinite-dimensional analog of a Jordan block, see, for example, 
a monograph by Jiang and Wang \cite{JW}.

Halmos initiated the study of transitive lattices and gave an example of 
transitive lattice consisting of seven subspaces in \cite{Ha2}. 
Harrison-Radjavi-Rosenthal \cite{HRR} constructed a transitive lattice  
consisting of six subspaces using the graph of an unbounded closed operator. 
Hadwin-Longstaff-Rosenthal found a transitive lattice of five  non-closed 
linear subspaces in \cite{HLR}.     
Any finite transitive lattice which consists of  $n$ subspaces of a Hilbert 
space $H$  
gives an indecomposable system of $n-2$ subspaces by withdrawing 
$0$ and $H$,  but the converse is not true. It is still unknown whether or not 
there exists a transitive lattice 
consisting of five subspaces.  Therefore it is also an interesting  problem to know whether there exists an indecomposable system of three subspaces in an infinite-dimensional Hilbert space.  
 
Throughout the paper a projection means an operator 
$e$ with $e^2 = e = e^*$ and 
an idempotent  means an operator $p$ with $p^2 = p$. 
The direct sum $\oplus$ is the orthogonal direct sum and 
$\oplus _{alg}$ is the algebraic direct sum. The subspace  
mostly means  closed subspace except the algebraic 
direct sum.

There seems to be interesting relations with the study of 
representations of $*$-algebras generated by idempotents by 
S. Kruglyak and Y. Samoilenko \cite{KS} and the study on sums of 
projections by S. Kruglyak, V. Rabanovich and Y. Samoilenko 
\cite{KRS}. But we do not know the exact implication, because 
their objects are different with ours. 

In finite dimensional case, the classification of four subspaces is 
described as  the classification of the representations  of the 
extended Dynkin diagram $D_4^{(1)}$. Recall that Gabriel \cite{G} 
listed Dynkin diagrams $A_n, D_n, E_6, E_7,E_8$ in his theory on 
finiteness of indecomposable representations of quivers.  
We discussed on indecomposable representations of 
quivers on infinite-dimensional Hilbert spaces 
 \cite{EW2}. We are also under the influence of subfactor theory by 
Jones \cite{J}. 

Our study also has a relation  with $C^*$-algebras generated by idempotents 
or projections. 
See Bottcher, Gohberg, Karlovich, Krupnik, Roch , Silbermann and 
Spittovsky \cite{BGKKRSS} , Hu and Xue \cite{HX} and references there. 

After we completed our paper, we noticed  that 
there might be  a connection with Kadison-Singer algebras introduced by 
Ge and Yuan \cite{GY1} and \cite{GY2}. 

The authors are supported by JSPS KAKENHI Grant number 23654053 and 
25287019.

\section{systems of $n$ subspaces}
We introduce  some basic definitions and facts on 
the relative position of $n$ subspaces
in a separable  Hilbert space.  
Let $H$ be a Hilbert space and $E_1, \dots E_n$ be $n$ subspaces 
in $H$.  Then we say that  ${\mathcal S} = (H;E_1, \dots , E_n)$  
is a system of $n$-subspaces in $H$ or an $n$-subspace system in $H$. 
Let ${\mathcal T} = (K;F_1, \dots , F_n)$  
be  another system of $n$-subspaces in a Hilbert space $K$. Then  
$\varphi : {\mathcal S} \rightarrow {\mathcal T}$ is called a 
homomorphism if $\varphi : H \rightarrow K$ is a bounded linear 
operator satisfying that  
$\varphi(E_i) \subset F_i$ for $i = 1,\dots ,n$. And 
$\varphi : {\mathcal S} \rightarrow {\mathcal T}$
is called an isomorphism if $\varphi : H \rightarrow K$ is 
an invertible (i.e., bounded  bijective) linear 
operator satisfying that  
$\varphi(E_i) = F_i$ for $i = 1,\dots ,n$. 
We say that systems ${\mathcal S}$ and ${\mathcal T}$ are 
{\it isomorphic} if there is an isomorphism  
$\varphi : {\mathcal S} \rightarrow {\mathcal T}$. This means 
that the relative positions of $n$ subspaces $(E_1, \dots , E_n)$ in $H$ 
and   $(F_1, \dots , F_n)$ in $K$ are same under disregarding angles. 
We say that systems ${\mathcal S}$ and ${\mathcal T}$ are 
{\it unitarily equivalent } if the above isomorphism 
$\varphi : H \rightarrow K$ can be chosen to be a unitary. 
This means that the relative positions of 
$n$ subspaces $(E_1, \dots , E_n)$ in $H$ 
and  $(F_1, \dots , F_n)$ in $K$ are same with preserving the angles
between the subspaces. We are interested in the relative position of 
subspaces up to isomorphism to study the still-remaining fundamental 
feature of the relative position after disregarding 
 \lq\lq the angles'' .  

We denote by 
$Hom(\mathcal S, \mathcal T)$ the set of homomorphisms of 
$\mathcal S$ to $\mathcal T$ and  
$End(\mathcal S) := Hom(\mathcal S, \mathcal S)$ 
the set of endomorphisms on $\mathcal S$. 

Let $G_2 = \mathbb Z/ 2\mathbb Z * \mathbb Z/ 2\mathbb Z =  
\langle a_1, a_2 \rangle$
be the free product of the cyclic groups of order two with generators 
$a_1$ and $a_2$. 
For two subspaces $E_1$ and $E_2$ of a Hilbert space $H$, 
let $e_1$ and $e_2$ be the projections onto $E_1$ and $E_2$.  
Then  $u_1 = 2e_1 -I$ and $u_2 = 2e_2 - I$ are self-adjoint unitaries.   
Thus there is a bijective correspondence between the set 
${\mathcal Sys}^2(H)$ of systems 
${\mathcal S} = (H;E_1,E_2)$ of two subspaces in a Hilbert space $H$ 
and the set $Rep(G_2,H)$ of unitary representations $\pi$ of $G_2$ on $H$ 
such that  $\pi(a_1) = u_1$ and $\pi(a_2) = u_2$.
Similarly let $G_n = \mathbb Z/ 2\mathbb Z * ... * \mathbb Z/ 2\mathbb Z$
be the $n$-times free product of the cyclic groups of order two. Then
there is a bijective correspondence between the set 
${\mathcal Sys}^n(H)$ of systems of $n$
subspaces in a Hilbert space $H$  and the set 
$Rep(G_n,H)$ of unitary representations of $G_n$ on $H$.  

\bigskip
\noindent
{\bf Example 1}.  Let $H = \mathbb C ^2$. Fix an angle $\theta$
with $0 < \theta < \pi /2$.  Put 
$E_1 = \{ \lambda(1,0) \ | \ \lambda \in \mathbb C\}$ and
$E_2 = \{ \lambda (cos\theta, sin\theta)\ | \ \lambda \in \mathbb C\}$.  Then
${\mathcal S}_1 = (H;E_1,E_2)$ is isomorphic to 
${\mathcal S}_2 = 
({\mathbb C}^2 ; {\mathbb C} \oplus 0, 0 \oplus {\mathbb C})$.
But the corresponding two unitary representations $\pi_1$ and 
$\pi_2$ are not similar, because 
$\frac{1}{2}(\pi_1(a_1) + 1) \frac{1}{2}(\pi_1(a_2) + 1)
 \not= 0$
and 
$\frac{1}{2}(\pi_2(a_1) + 1) \frac{1}{2}(\pi_2(a_2) + 1) =  0$.

We start with known facts to recall some notation. 
See \cite{EW1} for example. 

Let $H$ be a Hilbert space and $H_1$ and $H_2$ be 
two subspaces of $H$. We write  $H_1 \vee H_2 := \overline{H_1 + H_2}$ and 
$H_1 \wedge H_2:= H_1 \cap H_2$. Then the set of (closed) subspaces 
of $H$ forms a lattice under these operations $\vee$ and $\wedge$. 
Two subspaces $H_1$ and $H_2$ are said to be topologically 
complementary if $H = H_1 \vee H_2$ \ and  \ $H_1 \wedge H_2 = 0$. 
Two subspaces $H_1$ and $H_2$ are said to be algebraically 
complementary if $H = H_1 + H_2$ \ and  \ $H_1 \cap H_2 = 0$.

\begin{lemma}  Let $H$ be a Hilbert space and $H_1$ and $H_2$ be 
two subspaces of $H$.  Then the following are equivalent:
\begin{enumerate}
\item $H_1$ and $H_2$ are algebraically complementary, i.e., 
$H = H_1 + H_2$ \ and \ $H_1 \cap H_2 = 0$. 
\item There exists a closed subspace $M \subset H$ such that
$(H;H_1,H_2)$ is isomorphic to  $(H;M,M^{\perp})$
\item There exists an idempotent $P \in B(H)$ such that 
$H_1 = \Im P$ and $H_2 = \Im (1-P)$. 
\end{enumerate}
\label{lemma:decompose}
\end{lemma}

\begin{lemma}[\cite{EW1}] Let $H$ and $K$ be Hilbert spaces and 
$E \subset H$ and $F \subset K$ be closed subspaces of $H$ and $K$.  
Let $e \in B(H)$ and $f \in B(K)$ be the projections onto $E$ and $F$. 
Then the following are equivalent:
\begin{enumerate}
\item There exists an invertible operator $T: H \rightarrow K$
such that $T(E)=F$. 
\item There exists an invertible operator $T: H \rightarrow K$
such that $e =(T^{-1}fT)e$ and $f=(TeT^{-1})f$. 
\end{enumerate}
\end{lemma}

Using the above lemma, we can describe an isomorphism between 
two systems of $n$ subspaces in terms of operators only as follows:

\begin{cor}
Let $\mathcal S=(H;E_{1},\cdots,E_{n})$
and $\mathcal S^{\prime}
 =(H^{\prime};E_{1}^{\prime},\cdots ,E_{n}^{\prime})$
be two systems of n-subspaces. Let $e_i$ 
(resp. $e_{i}^{\prime}$) be the projection onto 
$E_i$ (resp. $E_{i}^{\prime}$) .
Then two systems $\mathcal S$ and $\mathcal S^{\prime}$ are 
isomorphic if and only if there exists an invertible operator
$T: H \rightarrow H^{\prime}$
such that $e_i =(T^{-1}e_{i}^{\prime}T)e_i$ and 
$e_{i}^{\prime} =(Te_iT^{-1})e_{i}^{\prime}$ for 
$i = 1, \dots, n$.  
\end{cor}

\noindent
{\bf Remark.} If there exists an invertible operator
$T: H \rightarrow H^{\prime}$
such that $e_{i}^{\prime}=Te_iT^{-1}$ for 
$i = 1, \dots, n$, then 
two systems $\mathcal S$ and $\mathcal S^{\prime}$ are 
isomorphic.  But the converse is not true as in Example 1.

\section{indecomposable systems}
In this section we  shall introduce a notion of indecomposable system,  
that is, a system which cannot be decomposed into a direct 
sum of smaller systems anymore. 

\bigskip
\noindent
{\bf Definition} (direct sum). 
Let ${\mathcal S} = (H;E_1, \dots , E_n)$ and 
 $\mathcal S^{\prime}=
 (H^{\prime};E_{1}^{\prime},\cdots,E_{n}^{\prime})$ be 
systems of $n$ subspaces in Hilbert spaces 
 $H$ and $H^{\prime}$.  Then their direct sum 
$\mathcal {S} \oplus\mathcal {S}^{\prime}$ is defined by 
\[
\mathcal {S} \oplus\mathcal {S}^{\prime}
:= (H\oplus H^{\prime};
E_{1}\oplus E_{1}^{\prime},\dots,E_{n}\oplus E_{n}^{\prime}).
\]

\bigskip
\noindent
{\bf Definition.} (indecomposable system). 
A system $\mathcal S=(H;E_{1},\dots,E_{n})$
of $n$ subspaces is called {\it decomposable} 
if the system $\mathcal S$ is isomorphic to 
a direct sum of two non-zero systems.  
A system $\mathcal S=(H;E_{1},\cdots,E_{n})$ is said to be 
{\it indecomposable} if it is not decomposable. 

\bigskip
\noindent
{\bf Example 2.} Let $H = \mathbb C ^2$. Fix an angle $\theta$
with $0 < \theta < \pi /2$.  Put 
$E_1 = \{ \lambda(1,0) \ | \ \lambda \in \mathbb C\}$ and
$E_2 = \{ \lambda (cos\theta, sin\theta)\ | \ \lambda \in \mathbb C \}$.  
Then
$(H;E_1,E_2)$ is isomorphic to 
\[ 
({\mathbb C}^2 ; {\mathbb C}\oplus 0, 0 \oplus {\mathbb C}) 
\cong (\mathbb C; \mathbb C, 0) \oplus (\mathbb C; 0, \mathbb C).
\]
Hence $(H;E_1,E_2)$ is decomposable. 
 
\bigskip
\noindent
{\bf Remark.} Let $e_1$ and $e_2$ be the
projections onto $E_1$ and $E_2$ in the Example above.  
Then the $C^*$-algebra 
$C^*(\{e_1,e_2\})$ generated by $e_1$ and $e_2$ is exactly $B(H)
\cong M_2(\mathbb C)$. Therefore 
the irreducibility of 
$C^*(\{e_1,e_2\})$ does {\it not} imply the 
indecomposability of $(H;E_1,E_2)$. Thus seeking an 
indecomposable system of subspaces is much more difficult and 
fundamental task than showing irreducibility of 
the $C^*$-algebra generated by the corresponding projections for 
the subspaces.   
\bigskip

We can characterize decomposability of systems inside the 
ambient Hilbert space as in \cite{EW1}

 Let $H$ be a Hilbert space and 
${\mathcal S} = (H;E_1, \dots , E_n)$  a system of $n$ subspaces.  
Then the following conditions are equivalent:
\begin{enumerate}
\item ${\mathcal S}$ is decomposable.
\item there exist non-zero closed
subspaces $H_1$ and $H_2$ of $H$ such that $H_1 + H_2 = H$,
$H_1 \cap H_2 = 0$ and $E_i = E_i \cap H_1 + E_i \cap H_2$
for any $i= 1,\ldots , n$.
\end{enumerate}

We give a condition of decomposability  in terms of 
endomorphism algebras for the systems.

We put $Idem(\mathcal S) := \{T \in End(\mathcal S) ; T = T^2 \}$.

Let $\mathcal S = (H;E_1,\ldots ,E_n)$ be a system
of $n$ subspaces in a Hilbert space $H$. 
Then $\mathcal S$ is indecomposable if and only if 
$Idem(\mathcal S) = \{0,I\}$.

Let $\mathcal S = (H;E_1,\ldots ,E_n)$ be a system
of $n$ subspaces in a Hilbert space $H$. 
Let $e_i$ be the projection of $H$ onto $E_i$ for $i = 1,\ldots ,n$.  
If $\mathcal S = (H;E_1,\ldots ,E_n)$
is indecomposable, then the $C^*(\{e_1,\ldots, e_n\})$ generated by
$e_1, \ldots, e_n$ is irreducible.  But the converse is not true.

\bigskip
\noindent
{\bf Definition.} Let $\mathcal S = (H;E_1,\ldots ,E_n)$ 
be a system of $n$ subspaces in a Hilbert space $H$. 
Let $e_i$ be the projection of $H$ onto $E_i$ for $i = 1,\ldots ,n$.  
We say that $\mathcal S$ is a {\it commutative} system 
if the $C^*(\{e_1,\ldots, e_n\})$ generated by
$e_1, \ldots, e_n$ is commutative. Be careful that commutativity 
is {\it not} an isomorphic invariant as shown in Example 1. But 
it makes sense that a system is isomorphic to a commutative 
system.

Let $\mathcal S = (H;E_1,\ldots ,E_n)$ 
be a system of $n$ subspaces in a Hilbert space $H$.
Assume that $\mathcal S$ is a commutative system. 
Then $\mathcal S$ is indecomposable if and only if 
$\dim H = 1$. Moreover each subset $\Lambda \subset \{1,\dots,n\}$ 
corresponds to a commutative system satisfying $\dim E_i = 1$ for 
$i \in \Lambda$ and $\dim E_i = 0$ for $i \notin \Lambda$. 
 
\bigskip
\noindent{\bf Example 3}. Let $H = \mathbb C ^2$.
Put $E_1 = \mathbb C \oplus 0$,
$E_2 = 0 \oplus \mathbb C$ and $E_3 = \{(x,x) \ x \in \mathbb C \}$.
Then $\mathcal S = (H;E_1,E_2,E_3)$ is indecomposable.
The system $\mathcal S$ is the lowest dimensional one 
among non-commutative indecomposable systems. In fact, 
the system $\mathcal S$ forms a double triangle 
in the sense below.  We see that 
distributive law fails: 
$$
 (E_1 \vee E_2) \wedge E_3 \not= (E_1 \wedge E_2) \vee (E_1 \wedge E_3).
$$

\bigskip
\noindent
{\bf Definition.} We say that a system $\mathcal S = (H;E_1,E_2,E_3)$ 
of three subspaces in a Hilbert space $H$ {\it forms a double triangle} if 
the family  $\{H,E_1,E_2,E_3,0 \}$ is a double triangle lattice,  
(which is also called a diamond), that is,
$$
E_i \vee E_j = H, \ \ \ \text{ and } \ \ \ E_i \wedge E_j = 0,  
\ \ (i \not= j , i, j = 1,2,3). 
$$
and each $E_i \not= H$, $E_i \not= 0$ . 
We remark that the distributive law fails in any double triangle. 
$$
 (E_1 \vee E_2) \wedge E_3 \not= (E_1 \wedge E_2) \vee (E_1 \wedge E_3).
$$
 
\bigskip
\noindent{\bf Example 4.} Let $G = \mathbb Z/ 2\mathbb Z * \mathbb Z/ 2\mathbb Z * \mathbb Z/ 2\mathbb Z =  
\langle a_1, a_2, a_3 \rangle$
be the free product of the cyclic groups of order two with three generators 
$a_1, a_2$ and $a_3$. Let $\lambda$ be the left regular representation of $G$ 
on $H = \ell^2(G)$. Then the reduced group $C^*$-algebra $C^*_r(G)$ is generatedby $\lambda_{a_1}$, $\lambda_{a_2}$ and $\lambda_{a_3}$. Since these three generators are self-adjoint unitaries, $e_i := (\lambda_{a_i} + I)/2$, ($i = 1,2,3$)
are projections. Let $E_i = \Im e_i$ . Then a system 
$\mathcal S = (H;E_1,E_2,E_3)$ 
of three subspaces forms a double triangle. In fact, 
let $x = \sum_g x_g \delta_g \in E_1 \cap E_2$. Since $e_i x = x$ , 
we have $\lambda_{a_i}x = x$ for 
$i = 1,2$. Therefore $x_{a_ig} = x_g$ 
for any $g \in G$. Since $\sum_h |x_h|^2 < \infty$, $x_g = 0$ for any $g$. 
Therefore $x = 0$. Hence $E_1 \cap E_2 = O$. The other conditions are similarly checked.

\bigskip
\noindent
{\bf Definition.} We say that a system $\mathcal S = (H;E_1,E_2,E_3)$ 
of three subspaces in a Hilbert space $H$ {\it forms a pentagon }
(with $E_3 \supset E_2$ )
 if 
the family  $\{ H,E_1,E_2,E_3,0 \} $ is a pentagon lattice 
(with $E_3 \supset E_2$ ), that is,
$$
E_1\vee E_2 = H, \ \  E_1 \wedge E_3 = 0,  
\text{ and } E_3 \supset E_2 \text{ with  } E_3 \not= E_2, 
$$
and each $E_i \not= H$, $E_i \not= 0$. 
We also say that $\mathcal S = (H;E_1,E_2,E_3)$  is a 
pentagon system. 

\bigskip
\noindent{\bf Example 5}. Let $K$ be a Hilbert space and 
$A: K \rightarrow K$ a bounded operator such that $A$ is one to one and 
$\Im A$ is dense in $K$ and not equal to $K$. Put  
$H = K \oplus K$, $E_1 = K \oplus 0$ and $E_2 = \{(x, Ax) | x \in K\}$. 
Let $M \not= 0$ be a finite-dimensional subspace of $K$ such that 
$M \cap \Im A = 0$. Put $E_3 = E_2 + (O,M)$. Then  
$\mathcal S = (H;E_1,E_2,E_3)$  is a 
pentagon system.

Recall that Halmos initiated the study of transitive lattices.  
A complete lattice of closed subspaces of a Hilbert space $H$
containing $0$ and $H$ is called 
{\it transitive} if every bounded operator on $H$ leaving each subspace 
invariant is a scalar multiple of the identity. Halmos gave an example of 
transitive lattice consisting of seven subspaces in \cite{Ha2}. 
Harrison-Radjavi-Rosenthal \cite{HRR} constructed a transitive lattice  
consisting of six subspaces using the graph of an unbounded operator.   
Any finite transitive lattice which consists of  $n$ subspaces  
gives an indecomposable system of $n$-$2$ subspaces  but the converse is 
not true. Following the study of transitive lattices, we shall introduce 
the notion of  transitive system.

\bigskip
\noindent{\bf Definition.} Let $\mathcal S = (H;E_1,\ldots ,E_n)$ 
be a system of $n$ subspaces in a Hilbert space $H$. Then we say that 
$\mathcal S$ is {\it transitive} if 
$End(\mathcal S) = {\mathbb C}I_H$.  Recall that  $\mathcal S$ is 
indecomposable if and only if $Idem(\mathcal S) = \{0,I\}$. Hence 
if $\mathcal S$ is transitive, then $\mathcal S$ is indecomposable.
But the converse is not true. In fact, 
the system
\[
{\mathcal S}_S = (H;K\oplus 0,0\oplus K,\graph A, \{(x,x) ; x \in K\}).
\]
of four subspaces associated with a unilateral shift $S$ 
as above is indecomposable but is not transitive, because 
$End(\mathcal S)$ contains $S \oplus S$.

\bigskip
\noindent{\bf Example 6}.(Harrison-Radjavi-Rosenthal \cite{HRR})
Let $K = \ell^2(\mathbb Z)$ and $H = K \oplus K$.
Consider a sequence $(\alpha _n)_n$ given by $\alpha _n = 1$ for
$n \le 0$ and $\alpha _n = exp((-1)^nn!)$ for $n \geq 1$.
Consider a bilateral weighted shift $S : \mathcal D_T \rightarrow K$ such
that $T(x_n)_n = (\alpha _{n-1}x_{n-1})_n$ with the domain
$\mathcal D_T = \{(x_n)_n \in  \ell^2(\mathbb Z) ; 
\sum_n |\alpha _nx_n|^2 < \infty \}$.
Let $E_1 = K \oplus 0$,
$E_2 = 0 \oplus K$, $
E_3 = \{(x,Tx) \in H ; x \in \mathcal D_T \}$ and
$E_4 = \{(x,x) \in H ; x \in K \}$.
Harrison, Radjavi and Rosenthal showed that 
$\{0, H, E_1,E_2,E_3,E_4\}$ is a transitive lattice. 
Hence the  system $\mathcal S = (H;E_1,E_2,E_3,E_4)$
of four subspaces in $H$ is transitive and in particular 
indecomposable.

It is easy to see the case of indecomposable systems 
of one subspace even in an infinite-dimensional Hilbert space.

Let $H$ be a Hilbert space and 
$\mathcal {S}=(H;E)$ a system of one subspace. 
Then $\mathcal{S}=(H;E)$ is indecomposable 
if and only if 
$\mathcal {S} \cong (\mathbb C;0)$ or 
$\mathcal {S} \cong (\mathbb C;\mathbb C)$.

Let $\mathcal {S}=(H;E)$ and $\mathcal {S}'=(H';E')$ be 
two systems of one subspace.  Then $\mathcal {S}$ and 
$\mathcal {S}'$ are isomorphic if and only if 
$\dim E = \dim E'$ and $\codim E = \codim E'$.

It is a well known fact that the relative position of two subspaces
$E_1$ and $E_2$ in a Hilbert space $H$  can be described 
completely up to unitary equivalence. 
 The Hilbert space $H$ is the direct sum of five
subspaces:
$$
H = (E_1\cap E_2) \oplus (\text{the rest}) \oplus (E_1\cap E_2^{\perp})
\oplus (E_1^{\perp} \cap E_2) \oplus (E_1^{\perp} \cap E_2^{\perp}).
$$
In the rest part,  $E_1$ and $E_2$ are in generic position and the
relative position is described only by \lq\lq the angles'' between them.
In fact the rest part is written as $K \oplus K$ for some subspace
$K$ and there exist two positive operators $c,s \in B(K)$ 
with null kernels with $c^2 + s^2 = 1$ such that 
\[
E_{1}=(E_{1}\cap E_{2})\oplus 
\Im \left( 
\begin{array}
{@{\,}cccc@{\,}} 1&0\\ 0&0
\end{array}
\right) \oplus 
(E_{1}\cap E_{2}^{\perp})\oplus 0 \oplus 0,
\]
and
\[
E_{2}=(E_{1}\cap E_{2}) \oplus 
\Im\left(
\begin{array}
{@{\,}cccc@{\,}}c^{2}&cs\\ cs&s^{2}
\end{array}
\right)
\oplus 0 \oplus (E_{1}^{\perp}\cap E_{2}) \oplus 0.
\]
By the functional calculus, there exists a unique positive operator
$\theta$, called the angle operator, such that
$c = \cos \theta \ \  \mbox{and} \ \ s = \sin \theta$ 
with $0 \leq \theta \leq \frac{\pi}{2}$. 
 We see  that 
the algebraic sum $E_1 + E_2$ is closed 
if and only if $scK + s^2K = K$ if and only if  $s$ is invertible.  
And $E_1 + E_2^{\perp}$ is closed if and only if 
$E_1^{\perp} + E_2$ is closed if and only if $c$ is invertible.  
We need the following fact: 

\begin{lemma}
Let 
$E_1$ and $E_2$ be two subspaces in a Hilbert space $H$. 
Let $P_i$ be the projection of $H$ onto  $E_i$. If 
$E_1 + E_2$ is closed, then 
$T:= (P_1 + P_2)|_{(E_1 + E_2)}: E_1 + E_2 \rightarrow E_1 + E_2$ is an 
onto invertible operator. 
\label{lemma:invertible}
\end{lemma}
\begin{proof}
Since $E_1 + E_2$ is closed, $s$ is invertible.  
Then it is easy to see that 
\[T= I \oplus 
\left(
\begin{array}
{@{\,}cccc@{\,}}I +c^{2} &cs\\ cs&s^{2}
\end{array}
\right) 
\oplus I \oplus I
\]
is invertible, because  the non-trivial component has the 
operator determinant
 $(I + c^2)s^2 -cscs = s^2$.
\end{proof}

 Let $\mathcal S = (H;E_1,E_2)$ be a system of two subspaces in 
a Hilbert space $H$.
Then $\mathcal S$ is indecomposable if and only if $\mathcal S$ is 
isomorphic to one of the
following four commutative systems:
$$
\mathcal S_1 = (\mathbb C; \mathbb C, 0), \ \  \mathcal S_2 = (\mathbb C; 0, \mathbb C), \\
\mathcal S_3 = (\mathbb C; \mathbb C, \mathbb C), \ \  \mathcal S_4 = (\mathbb C; 0, 0).
$$

\section{Brenner type decomposition }
We introduce a Brenner type decomposition which is a generalization 
of a Brenner decomposition of a system of three subspaces in a 
finite dimensional Hilbert space. 

\bigskip
\noindent
{\bf Definition.} Let  $\mathcal S = (H;E_1,E_2,E_3)$ be a 
system of three subspaces in a Hilbert space $H$. Then 
$\mathcal S$ is said to have a {\it Brenner type decomposition} if 
$\mathcal S$  is isomorphic to a system 
${\mathcal T} = (H;F_1,F_2,F_3 )$ satisfying  that there exist 
subspaces $S, N_1,N_2, N_3, M_1, M_2, M_3, Q, L$ of $H$ such  
that $(Q; Q_1,Q_2, Q_3)$ forms a  double triangle and 
\begin{align*}
  H & = S \oplus N_1 \oplus N_2 \oplus N_3 \oplus 
      M_1 \oplus M_2 \oplus M_3 \oplus Q \oplus L \\
F_1 & = S \oplus 0_{ \ \ } \oplus N_2 \oplus N_3 \oplus 
      M_1 \oplus 0_{ \ \ } \oplus 0_{ \ \ } \oplus Q_1 \oplus 0_{ \ \ }\\  
F_2 & = S \oplus N_1 \oplus 0_{ \ \ } \oplus N_3 \oplus 
      0_{ \ \ } \oplus M_2 \oplus 0_{ \ \ } \oplus Q_2 \oplus 0_{ \ \ } \\  
F_3 & =  S \oplus N_1 \oplus N_2 \oplus 0_{ \ \ } \oplus 
     0_{ \ \ } \oplus 0_{ \ \ } \oplus M_3 \oplus Q_3 \oplus 0_{ \ \ } \\  
\end{align*} 

\begin{prop}
 Let  $\mathcal S = (H;E_1,E_2,E_3)$ be a 
system of three subspaces in a Hilbert space $H$. 
Assume that $\mathcal (H;E_1,E_2,E_3)$ forms a double triangle.
Then the followings are equivalent: 
\begin{enumerate}
\item Linear sums $E_i + E_j$ are closed for any 
$i,j \in  \{1,2,3\}$ with 
$i \not= j$. \\
\item $\mathcal S$ is isomorphic to  a typical form, i.e. 
$$
(H; H_1,H_2, H_3) \cong 
(K \oplus K ; K \oplus 0, 0 \oplus K, \{(x,x) \ | x \in K\})
$$
for some Hilbert space $K$. 
\end{enumerate}
\label{prop:typical}
\end{prop}
\begin{proof}It is trivial that (2) implies (1). Conversely, 
assume (1). Since $H = E_1 + E_2$ and $E_1 \cap E_2 = 0$, we 
may and do assume that $E_2 = E_1^{\perp}$ up to isomorphism. 
Apply two subspace theorem for $E_1$ and $E_3$. Since 
$E_1 \cap E_3 = 0$ and $E_1^{\perp} \cap E_3^{\perp} = 0$, 
The Hilbert space $H$ is the direct sum of three
subspaces:
$$
H =  (K \oplus K) \oplus (E_1\cap E_3^{\perp})
\oplus (E_1^{\perp} \cap E_3) .
$$
for some subspace $K$ and there exist 
two positive operators $c,s \in B(K)$ 
with null kernels with $c^2 + s^2 = 1$ such that 
\[
E_{1}= 
\Im \left( 
\begin{array}
{@{\,}cccc@{\,}} 1&0\\ 0&0
\end{array}
\right) \oplus 
(E_{1}\cap E_{3}^{\perp})\oplus 0 ,
\]
and
\[
E_{3}= 
\Im\left(
\begin{array}
{@{\,}cccc@{\,}}c^{2}&cs\\ cs&s^{2}
\end{array}
\right)
\oplus 0 \oplus (E_{1}^{\perp}\cap E_{3}) .
\]
Since $E_2 \cap E_3 = 0$, we have that $E_{1}^{\perp}\cap E_{3} = 0$. 
Since $E_2 \vee E_3 = H$, we have that  $E_{1}\cap E_{3}^{\perp} = 0$. 
Moreover 
\[
\Im\left(
\begin{array}
{@{\,}cccc@{\,}}c^{2}&cs\\ cs&s^{2}
\end{array}
\right) 
= \{\binom{cz}{sz} \ | \ z \in K\},
\] 
because $c+s$ is invertible.

Since $E_1 + E_3$ is closed,  
$s$ is invertible.  
Since $E_1^{\perp} + E_3 = E_2 + E_3$ is closed, $c$ is invertible.
Consider an invertible operator 
\[
T = \left(
\begin{array}
{@{\,}cccc@{\,}}c^{-1}&0\\ 0&s^{-1}
\end{array}
\right)
\]
Then $TE_1 = K \oplus 0$, $TE_2 = 0 \oplus K$ and 
$TE_3 = \{\binom{z}{z} \ | \ z \in K\}$. 
This completes the proof. 
\end{proof}

We need the following Theorem after \cite[Corollary \ 4.1]{Fe}  
by Feshchenko who studies 
closedness of the sum of $n$ subspaces of a Hilbert space. 
Let $H_1, \dots, H_n$ be subspaces of a Hilbert space.  
Then  $H_1, \dots, H_n$ are said to be linearly independent if 
for any $x_i \in H_i$ ($i = 1,\dots, n$), if 
$x_1 + \dots x_n = 0$, then $x_1 = \dots =x_n = 0$. They are 
linearly independent if and only if 
the representation $x = x_1 + \dots +x_n$ for 
$x_i \in H_i$ ($i = 1,\dots, n$) is unique if and only if 
$$
H _i \cap (\sum_{\{ j ; j \not= i \}} H_j )= 0
$$
for any $i = 1,\dots, n$.

\begin{thm}[\cite{Fe}] Let $H_1, \dots, H_n$ be linear independent 
subspaces of a Hilbert space 
$H$. If $H = H_1 + \dots + H_n$, then for any collection of subscripts 
$i(1), \dots, i(k)$, the sum $H_{i(1)} + \dots + H_{i(k)}$ is closed. 
 \label{thm:closedness}
\end{thm}

Using the Feshchenko's Theorem above, we can extend 
Lemma \ref{lemma:decompose} to $n$-subspaces. 

\begin{thm}
Let $H_1, \dots, H_n$ be $n$-subspaces of a Hilbert space $H$. 
Then the following are equivalent:
\begin{enumerate}
\item $H = H_1 + \dots + H_n$ and 
$H_1, \dots, H_n$ are linearly independent.  
\item $H$ is isomorphic to an outer orthogonal sum  
$H_1 \oplus H_2 \dots \oplus H_n$, i.e., there exists
an invertible operator $T: H \rightarrow 
H_1 \oplus H_2 \dots \oplus H_n$ such that $T(H_i) = 0 \oplus H_i \oplus 0$. 
\end{enumerate}
\label{thm:sum}
\end{thm}

\begin{proof}
Assume (1). By the Feshchenko's Theorem above, 
$H_1 + H_2$ is closed. Since $H_1$ and $H_2$ are linearly 
independent, $H_1 + H_2$ is isomorphic to an outer orhtogonal sum 
$H_1 \oplus H_2$ by Lemma \ref{lemma:decompose}. 
Since $(H_1 \oplus H_2) + H_3$ is closed by the Feshchenko's theorem 
and $(H_1 \oplus H_2)$ and $H_3$ are linearly independent,
 $(H_1 \oplus H_2) + H_3$ is isomorphic to an outer orhtogonal sum 
$H_1 \oplus H_2 \oplus H_3$ by Lemma \ref{lemma:decompose}. Inductively 
we can show (2). The converse is clear.  
\end{proof}

The failure of the ditributive law is measured by the inclusions:
\[
    ( (E_i \wedge E_j) \vee (E_i \wedge E_k)) \subset 
    (E_i \wedge (E_j \vee E_k))
\]

Therefore the finite dimensonality of its quotient space is 
a slight generalization of the finite dimensionality of the 
ambient space $H$. 

\begin{thm}
 Let  $\mathcal S = (H;E_1,E_2,E_3)$ be a 
system of three subspaces in a Hilbert space $H$.
Then the followings are equivalent: 
\begin{enumerate}
\item Linear sums $E_i + E_j$ and $(E_i \cap E_k) + (E_j \cap E_k)$ 
are closed for $i,j,k \in  \{1,2,3\}$ with $i \not= j \not= k \not= i$ 
and  the quotient space 
$ (E_3 \wedge (E_1 \vee E_2))/ ( (E_3 \wedge E_1) \vee (E_3 \wedge E_2))$ 
is finite-dimensional.\\
\item $\mathcal S$ has a Brenner type decomposition with a finite-dimensional 
double triangle part $Q$. 
\end{enumerate}

Moreover if these equivalent conditions are satisfied, then 
the double triangle part $Q$ is isomorphic to  a typical form, i.e. 
$$
(Q; Q_1,Q_2, Q_3) \cong 
(K \oplus K ; K \oplus 0, 0 \oplus K, \{(x,x) \ | x \in K\})
$$
for some Hilbert space $K$. 

\end{thm}
\begin{proof}It is trival that (2) implies (1). Conversely, 
assume (1). 
 Let 
\[
Q_3 = (E_3 \wedge (E_1 \vee E_2)) 
\cap ( (E_3 \wedge E_1) \vee (E_3 \wedge E_2))^{\perp}
\]
Then $Q_3$ is finite-dmensional by the assumption and 
\[
(E_3 \wedge (E_1 \vee E_2)) 
= ( (E_3 \wedge E_1) \vee (E_3 \wedge E_2)) \oplus Q_3
\]

Let $P_i$ be the projection of $H$ onto $E_i$.  Since 
$E_1 + E_2$ is closed, 
$T:= (P_1 + P_2)|_{(E_1 + E_2)}: E_1 + E_2 \rightarrow E_1 + E_2$ is an 
onto invertible operator by Lemma \ref{lemma:invertible}. 
Put $A_1 = P_1T^{-1}$ and $A_2 = P_2T^{-1}$. 
Then $A_1 + A_2 =id|_{E_1 + E_2}$. Put $Q_1 := A_1(Q_3)\subset E_1$ and 
$Q_2 := A_2(Q_3)\subset E_2$. Then $Q_1$ and $Q_2$ are finite-dimensional. 
For any $q_3 \in Q_3$, put $q_1 = A_1q_3 \in Q_1$ and 
$q_2 = A_2q_3 \in Q_2$. Then $q_1 + q_2 = q_3$.  Let $Q := Q_1 + Q_2$. 
Then 
\[
Q = Q_1 + Q_2 = Q_2 + Q_3 = Q_3 + Q_1. 
\]
Moreover $Q_2 \cap Q_3 = 0$. In fact, $Q_2 \cap Q_3 \subset E_2 \cap E_3$ 
and $Q_2 \cap Q_3 \subset Q_3 \subset (E_2 \cap E_3)^{\perp}$.  
Similarly we have  $Q_1 \cap Q_3 = 0$. Let $q \in Q_1 \cap Q_2$. 
Then there exists $q_3 \in Q_3$ such that $q = A_1q_3$ and 
\[
q_3 = A_1q_3 + A_2q_3 = q + A_2q_3 \in Q_2 + Q_2 = Q_2
\]
Thus $q_3 \in Q_3 \cap Q_2 = 0$. Hence $q = A_1q_3 = 0$.
This shows that $Q_1 \cap Q_2 = 0$. Therefore 
$(Q; Q_1,Q_2, Q_3)$ forms a  double triangle.

We shall show that 
\[
(E_1 \cap (E_2 + E_3)) 
= ( (E_1 \cap  E_2) + (E_1 \cap E_3)) \oplus_{alg} Q_1
\]
Since $Q_1 \subset E_1$ and 
$Q_1 \subset Q_2 + Q_3 \subset (E_2 + E_3)$, 
\[
(E_1 \cap (E_2 + E_3)) 
\supset  ( (E_1 \cap  E_2) + (E_1 \cap E_3)) + Q_1
\]
Conversely let $x_1 \in (E_1 \cap (E_2 + E_3))$. Then there 
exist $x_2 \in E_2$ and $x_3 \in E_3$ such that $x_1 = x_2 + x_3$. 
Since 
\[
x_3 = x_1 -x_2 \in E_3 \cap (E_1 + E_2) 
= ( (E_3 \cap E_1) + (E_3 \cap E_2)) + Q_3,
\]
there exist $y_1 \in E_3 \cap E_1$, $y_2 \in E_3 \cap E_2$ 
and $q_3 \in Q_3$ such that $x_3 = y_1 + y_2 + q_3$. 
Since $Q_3 \subset Q_1 + Q_2$, there exist $q_1 \in Q_1$ and 
$q_2 \in Q_2$ such that $q_3 = q_1 + q_2$.  Then we have that 
\[
x_1 -x_2 = x_3 =  y_1 + y_2 + q_3 = y_1 + y_2 +  q_1 + q_2.
\]
Put 
\[
 z_{12} := x_1 -y_1 - q_1 = y_2 + x_2 + q_2 \in  E_1 \cap E_2 . 
\]
Then 
\[
x_1 = q_1 + y_1 + z_{12} \in Q_1 + E_3 \cap E_1 + E_1 \cap E_2
\]
This implies that 
\[
(E_1 \cap (E_2 + E_3)) 
\subset  ( (E_1 \cap  E_2) + (E_1 \cap E_3)) + Q_1
\]
We shall show that 
$((E_1 \cap  E_2) + (E_1 \cap E_3)) \cap Q_1 = 0. $
Let $q_1 \in ((E_1 \cap  E_2) + (E_1 \cap E_3)) \cap Q_1$. 
Then there exist $y \in E_1 \cap  E_2$ and $z \in E_1 \cap E_3$ 
such that $q_1 = y + z$. Since $q_1 \in Q_1$, there exists 
$q_3 \in Q_3$ such that $ q_1 =A_1q_3$. Put $q_2 = A_2q_3$. 
Then $q_3 = q_1 + q_2$. Hence $y + z = q_1 = q_3 -q_2$. Put
\[
s := z - q_3 = -y - q_2  \in E_3 \cap E_2.
\]
Then $q_3 = z -s  \in (E_3 \cap E_1 ) + (E_3 \cap E_2)$. 
Hence $q_3 \in Q_3 \cap  ((E_3 \cap E_1 ) + (E_3 \cap E_2)) = 0$. 
Thus $ q_1 =A_1q_3 = 0$. 
Therefore we have that 
\[
(E_1 \cap (E_2 + E_3)) 
= ( (E_1 \cap  E_2) + (E_1 \cap E_3)) \oplus_{alg} Q_1
\]
Similarly we have that 
\[
(E_2 \cap (E_1 + E_3)) 
=  ( (E_2 \cap  E_1) + (E_2 \cap E_3)) \oplus_{alg} Q_2
\]
Put 
\[
M_1 := E_1 \cap (E_1 \cap (E_2 + E_3))^{\perp}
\]
\[
M_2 := E_2 \cap (E_2 \cap (E_3 + E_1))^{\perp}
\]
\[
M_3 := E_3 \cap (E_3 \cap (E_1 + E_2))^{\perp}
\]
Then we have that 
\[
E_1 = M_1 \oplus (E_1 \cap (E_2 + E_3)), \ \ 
E_2 = M_2 \oplus (E_2 \cap (E_3 + E_1))
\]
and 
\[
E_3 = M_3 \oplus (E_3 \cap (E_1 + E_2))
\]
Put $S := E_1 \cap E_2 \cap E_3$ and 
\[
N_1 := E_2 \cap E_3 \cap S^{\perp}, \ \ 
N_2 := E_3 \cap E_1 \cap S^{\perp} \ \ \text{ and }
N_3 := E_1 \cap E_2 \cap S^{\perp}. 
\]
Then we have that
\[
E_2 \cap E_3 = S \oplus N_1, \ \ 
E_3 \cap E_1 = S \oplus N_2  \ \ \text{ and }
E_1 \cap E_2 = S \oplus N_3. 
\]
Put $L := (E_1 + E_2 + E_3)^{\perp} \cap H$.  
Moreover 
\[
E_1 = M_1 + ( (E_1 \cap  E_2) + (E_1 \cap E_3)) +  Q_1 
    = S +  N_2 + N_3 + M_1 + Q_1
\]
Similarly we also have that 
\[
E_2 = S +  N_1 + N_3 + M_2 + Q_2, \ \ \text{ and }
E_3 = S  + N_1 + N_2 + M_3 + Q_3.
\]
Therefore 
\[
E_1 + E_2 + E_3 = S + N_1 + N_2 + N_3 +M_1 + M_2 + M_3 + Q
\]
and 
\[
H = (E_1 + E_2 + E_3) + L = S + N_1 + N_2 + N_3 +M_1 + M_2 + M_3 + Q + L.
\]
Finally we shall show that the linear sum of the right-hand side is 
in fact an algebraic direct sum. We need to show that 
$S, N_1, N_2 , N_3, M_1, M_2, M_3, Q$ and $L$ 
are linearly indepenent. Let 
\[
s + n_1 + n_2 + n_3 + m_1 + m_2 + m_3 + q_1 + q_2 + \ell = 0
\]
for $s \in S, n_1 \in N_1, n_2 \in N_2, n_3 \in N_3, 
m_1 \in M_1, m_2 \in M_2, m_3 \in M_3, q_1 \in Q_1, q_2 \in Q_2$ 
and $\ell \in L$. Then it is clear that $\ell = 0$. 
Therefore  
\[
-m_3 = (n_2 +m_1 +  q_1) + ( n_1 + m_2 + q_2) + (n_3 + s) 
\in E_1 + E_2 + E_1 \cap E_2 \subset E_1 + E_2. 
\]
Therefore $m_3  \in M_3 \cap (E_3 \cap (E_1 + E_2)) = 0$. 
Thus $m_3 = 0$. Since 
\[
q_1 + q_2 = q_2' + q_3' = q_3'' + q_1''
\]
for some $q _2 \in Q_2$, $q_3', q_3'' \in Q_3$ and $q_1'' \in Q_1$,. 
we similarly have that $m_1 = m_2 = 0$. 
Hence 
\[
s + n_1 + n_2 + n_3 + q_1 + q_2 = 0. 
\]
Put $w := n_1 + q_2 = -(n_2 + n_3 +q_1) \in E_2 \cap E_1$. 
Then 
\[
q_2 = w -n_1  \in (E_2 \cap E_1) + (E_2 \cap E_3)
\]
Since $q_2 \in Q_2 \cap ((E_2 \cap E_1) + (E_2 \cap E_3)) = 0$, 
we have that $q_2 = 0$. Similarly we have that $q_1 = 0$. 
Therefore $s + n_1 + n_2 + n_3 = 0$.  Since 
\[
E_1 \cap E_2 \ni s + n_3 = -n_1 -n_2 \in E_3, 
\]
$s + n_3 \in (E_1 \cap E_2) \cap E_3 = S$. Thus $n_3 \in S \cap N_3 = 0$. 
Therefore $n_3 = 0$. Similarly we have that $n_1 = n_2 = 0$. Hence 
$s = 0$. 
Finally Theorem \ref{thm:sum}
 implies the conclusion. The rest is clear by Proposition \ref{prop:typical}.

\end{proof}

As a Corollary, we get the original Brenner's theorem. 

\begin{cor}[\cite{B}, \cite{MS}]
Let  $\mathcal S = (H;E_1,E_2,E_3)$ be a 
system of three subspaces in a finite dimensional Hilbert space $H$. 
Then 
$\mathcal S$ has a Brenner type decomposition. 
\end{cor}

\noindent
{\bf Remark.} Even if an ambient space $H$ is finite-dimensional, 
a double triangle part is not uniquely determined in a 
Brenner type decomposition. In fact, 
let $H = {\mathbb C} \oplus {\mathbb C} \oplus {\mathbb C}$, 
$E_1 = {\mathbb C} \oplus 0 \oplus {\mathbb C}$, 
$E_2 = 0 \oplus {\mathbb C}  \oplus {\mathbb C}$ and 
$E_3 = \{x(1,1,1) \in H \  | \ x \in {\mathbb C}\}$. 
Put $N_3 = 0 \oplus 0 \oplus {\mathbb C}$. 
Put $S =N_1 =N_2 = M_1= M_2= M_3 =L = 0$. 
Let $Q_3 = E_3$, $Q_1 = \{x(1,0,1/2) \in H \  | \ x \in {\mathbb C}\}$ 
and $Q_2 = \{x(0,1,1/2) \in H \  | \ x \in {\mathbb C}\}$
This gives a Brenner type decomposition.  
We have another Brenner type decompositin by 
$Q_3' = E_3$, $Q_1' = \{x(1,0,1/3) \in H \  | \ x \in {\mathbb C}\}$ 
and $Q_2' = \{x(0,1,2/3) \in H \  | \ x \in {\mathbb C}\}$ and 
the others are the same as the first one. 
Since $Q := Q_1 + Q_2 \not= Q' := Q_1' + Q_2'$, they provide 
two kinds of Brenner type decompositions.

Let $\mathcal S = (H;E_1,E_2,E_3)$ be a 
system of three subspaces which has a Brenner type decomposition. 
Then it is clear that 
 for any $i,j,k = 1,2,3$ with $i \not= j$, $j \not= k$ and $k \not= i$,
$$
\overline{E_i + E_j } + E_k \  \text{ and } \ (E_i \cap E_j) + E_k
$$ 
are closed in $H$. We shall show that this topological property 
characterize a 
system of three subspaces 
which have a Brenner type decomposition. 

\bigskip
\noindent
{\bf Example 7.}
Let $\mathcal S = (H;E_1,E_2,E_3)$ be a system 
of three subspaces in a Hilbert space $H$. 
If  $\mathcal S = (H;E_1,E_2,E_3)$ forms a pentagon 
(with $E_1 \supset E_2$ ), then 
neither $E_1 + E_3$ nor $E_2 + E_3$ are closed. Therefore 
$\overline{E_1 + E_2 } + E_3 = E_1 + E_2$ is not closed and 
$(E_1 \cap E_2) + E_3 = E_2 + E_3$ is not closed. Hence 
$\mathcal S = (H;E_1,E_2,E_3)$ does not have 
a Brenner type decomposition. 
Thus this closedness property excludes pentagons to have a 
Brenner type decomposition.

We shall split out a distributive part and a double triangle part 
step by step. 

\begin{lemma}
Let $\mathcal S = (H;E_1,E_2,E_3)$ be a system 
of three subspaces in a Hilbert space $H$. Suppose that 
$(E_1 \cap E_2) + E_3$ is closed. Then 
there exist systems $\mathcal S ' = (H';E_1',E_2',E_3')$ and 
$\mathcal S '' = (H'';E_1'',E_2'',E_3'')$ of three subspaces such that 
\begin{enumerate}
\item $\mathcal S = (H;E_1,E_2,E_3) 
\cong (H';E_1',E_2',E_3') \oplus (H'';E_1'',E_2'',E_3'')$.
\item $ E_1'' \cap E_2'' = 0$
\item $E_3' \subset E_1' = E_2'$ \ \ \text {(\lq\lq distributive component ")} 
\end{enumerate}
\label{lemma:cap-zero}
\end{lemma}
\begin{proof}
Consider two subspace decomposition for $F:= E_1 \cap E_2$ and $E_3$. 
Since $(E_1 \cap E_2) + E_3$ is closed, we may and do assume that 
there exits no angle part for $F:= E_1 \cap E_2$ and $E_3$  up 
to isomorphism. Therefore we have the following decomposition: 
\begin{align*} 
H & = (E_1 \cap E_2) \cap E_3  \oplus (E_1 \cap E_2) \cap E_3^{\perp} 
    \oplus  (E_1 \cap E_2)^{\perp} \cap E_3    
    \oplus  (E_1 \cap E_2)^{\perp} \cap E_3^{\perp} \\
E_1 \cap E_2 & = (E_1 \cap E_2 \cap E_3) \oplus (E_1 \cap E_2) \cap E_3^{\perp}
    \oplus \ \ \ \ \ \ \ \ \ \ 0 \ \ \ \ \ \ \ \ \  \oplus 
\ \ \ \ \ \ \ \ \ 0  \\
E_3 & = (E_1 \cap E_2 \cap E_3) \oplus \ \ \ \ \ \ \ \ \ 0 
\ \ \ \ \ \ \ \ \  
\oplus (E_1 \cap E_2)^{\perp} \cap E_3   \oplus \ \ \ \ \ \ \ \ \ \ 0
\end{align*}
Then put $H' := E_1 \cap E_2$ and $H'' := (E_1 \cap E_2)^{\perp}$. 
Consider corresponding decompositions for $H = H' \oplus H''$. 
By taking intersections with  $H'$ and $H''$, let 
\[
E_1' = E_1 \cap (E_1 \cap E_2) = E_1 \cap E_2, \ \ 
E_2' = E_2 \cap (E_1 \cap E_2) = E_1 \cap E_2, \ \ 
\]
and 
\[
       E_3' = E_3 \cap (E_1 \cap E_2) = E_1 \cap E_2 \cap E_3 
\]
\begin{align*}
E_1'' & = E_1 \cap (E_1 \cap E_2)^{\perp}, \ \ \\
E_2'' & = E_2 \cap (E_1 \cap E_2)^{\perp} , \ \ \\
E_3'' & = E_3 \cap (E_1 \cap E_2)^{\perp}
\end{align*}
Then clearly we have that 
\[
E_1' + E_1'' = E_1 \ \ E_2' + E_2'' = E_2 \ \ \text{ and } E_3' + E_3'' = E_3
\]
Moreover we have that $E_3' \subset E_1' = E_2'$ and $E_1'' \cap E_2'' = 0$. 
 
\end{proof}

\begin{lemma}
Let $\mathcal S = (H;E_1,E_2,E_3)$ be a system 
of three subspaces in a Hilbert space $H$. Suppose that 
$(E_1 \vee E_2) + E_3$ is closed. Then 
there exist systems $\mathcal S ' = (H';E_1',E_2',E_3')$ and 
$\mathcal S '' = (H'';E_1'',E_2'',E_3'')$ of three subspaces such that 
\begin{enumerate}
\item $\mathcal S = (H;E_1,E_2,E_3) 
\cong (H';E_1',E_2',E_3') \oplus (H'';E_1'',E_2'',E_3'')$.
\item $ E_1'' \vee E_2'' = H''$
\item $E_1'= E_2'=0 \subset E_3'$ \ \ \text {(\lq\lq distributive component ")} \end{enumerate}
\label{lemma:vee}
\end{lemma}
\begin{proof}
Consider two subspace decomposition for $F:= E_1 \vee E_2$ and $E_3$. 
Since $(E_1 \vee E_2) + E_3$ is closed, we may and do assume that 
there exits no angle part for $F:= E_1 \vee E_2$ and $E_3$  up 
to isomorphism. Therefore we have the following decomposition: 
\begin{align*} 
H & = (E_1 \vee E_2) \cap E_3  \oplus (E_1 \vee E_2) \cap E_3^{\perp} 
    \oplus  (E_1 \vee E_2)^{\perp} \cap E_3    
    \oplus  (E_1 \vee E_2)^{\perp} \cap E_3^{\perp} \\
E_1 \vee E_2 & = (E_1 \vee E_2) \cap E_3 
\oplus (E_1 \vee E_2) \cap E_3^{\perp}
    \oplus \ \ \ \ \ \ \ \ 0 \ \ \ \ \ \ \ \ \oplus \ \ \ \ \ \ \ \ 0  \\
E_3 & = (E_1 \vee E_2) \cap E_3 \oplus \ \ \ \ \ \ \ \ \ 0 \ \ \ \ \ \ \ \ 
\oplus (E_1 \vee E_2)^{\perp} \cap E_3   \oplus \ \ \ \ \ \ \ \ 0 
\end{align*}
Then put $H' := (E_1 \vee E_2)^{\perp}$ and $H'' := (E_1 \vee E_2)$. 
Consider corresponding decompositions for $H = H' \oplus H''$. 
By taking intersections with  $H'$ and $H''$, let 
\[
E_1'  = (E_1 \vee E_2)^{\perp} E_1 = 0, \ \ 
E_2'  = (E_1 \vee E_2)^{\perp} E_2 = 0, \ \ 
       E_3' = (E_1 \vee E_2)^{\perp} \cap E_3 
\]
\begin{align*}
E_1'' & = (E_1 \vee E_2) \cap E_1 = E_1, \ \ \\
E_2'' & = (E_1 \vee E_2) \cap E_2 = E_2, \ \ \\
E_3'' & = (E_1 \vee E_2) \cap E_3,
\end{align*}
Then clearly we have that 
\[
E_1' + E_1'' = E_1 \ \ E_2' + E_2'' = E_2 \ \ \text{ and } E_3' + E_3'' = E_3
\]
Moreover we have that $E_1' = E_2'=  0 \subset E_3'$ and $E_1'' \vee E_2'' = H''$. 
 
\end{proof}

\begin{lemma}  Let $H$ be a Hilbert space and $H_1$ and $H_2$ be 
two closed subspaces of $H$. Assume that 
$H = H_1 + H_2$ \ and  \ $H_1 \cap H_2 = 0$. Let 
$F_i$ be an algebraic linear subspace of $H_i$ for $i = 1,2$. 
Consider their algebraic linear sum  $F := F_1 + F_2$ in $H$. 
Then $F$ is closed if and only if $F_1$ and $F_2$ are closed. 
\label{lemma:closedness}
\end{lemma}
\begin{proof}
There exists a closed subspace $M \subset H$ such that 
$(H;H_1,H_2)$ is isomorphic to  $(H;M,M^{\perp})$. Then 
the statement is redeuced to this case where the statement is 
clear. 
\end{proof}

By the Lemma above, we have immediately have the followings:

\begin{lemma}
Let $\mathcal S = (H;E_1,E_2,E_3)$ be a system 
of three subspaces in a Hilbert space $H$. Suppose that 
there exist systems $\mathcal S ' = (H';E_1',E_2',E_3')$ and 
$\mathcal S '' = (H'';E_1'',E_2'',E_3'')$ of three subspaces such that 
\[
\mathcal S = (H;E_1,E_2,E_3) 
\cong (H';E_1',E_2',E_3') \oplus (H'';E_1'',E_2'',E_3'').
\]
Then the followings hold:
\begin{enumerate}
\item 
$(E_1 \cap E_2) + E_3$ is closed in $H$ if and only if 
$(E_1' \cap E_2') + E_3'$ is closed in $H'$ and 
$(E_1'' \cap E_2'') + E_3''$ is closed in $H''$. 
\item
$(E_1 \vee E_2) + E_3$ is closed in $H$ if and only if 
$(E_1' \vee E_2') + E_3$ is closed in $H'$ and 
$(E_1' \vee E_2') + E_3'$ is closed in $H''$. 
\item 
$ E_1 \cap E_2 = 0$ if and only if 
$ E_1' \cap E_2' = 0$ and $ E_1'' \cap E_2'' = 0$. 
\item
$E_1 \vee E_2 = H$ if and only if 
$E_1' \vee E_2' = H'$ and $E_1'' \vee E_2 ''= H''$
\end{enumerate}
\label{lemma:decomposed}
\end{lemma}

\begin{thm}
 Let  $\mathcal S = (H;E_1,E_2,E_3)$ be a 
system of three subspaces in a Hilbert space $H$.
Then the followings are equivalent: 
\begin{enumerate}
\item Linear sums $(E_i \vee E_j) + E_k$ and $(E_i \cap E_j) + E_k$ 
are closed for $i,j,k \in  \{1,2,3\}$ with $i \not= j \not= k \not= i$. 

\item $\mathcal S$ has a Brenner type decomposition. 
\end{enumerate}
\end{thm}
\begin{proof}It is trival that (2) implies (1). Conversely, 
assume (1). 
Since 
$(E_1 \cap E_2) + E_3$ is closed, 
there exist systems $\mathcal S ' = (H';E_1',E_2',E_3')$ and 
$\mathcal S '' = (H'';E_1'',E_2'',E_3'')$ of three subspaces such that 
\begin{enumerate}
\item $\mathcal S = (H;E_1,E_2,E_3) 
\cong (H';E_1',E_2',E_3') \oplus (H'';E_1'',E_2'',E_3'')$.
\item $ E_1'' \cap E_2'' = 0$
\item $E_3' \subset E_1' = E_2'$ \ \ \text {(\lq\lq distributive component ")} 
\end{enumerate}

Since 
$(E_1'' \vee E_2'') + E_3''$ is closed. Then 
there exist systems $\mathcal S ''' = (H''';E_1''',E_2''',E_3''')$ and 
$\mathcal S '''' = (H'''';E_1'''',E_2'''',E_3'''')$ 
of three subspaces such that 
\begin{enumerate}
\item $\mathcal S'' = (H'';E_1'',E_2'',E_3'') 
\cong (H''';E_1''',E_2''',E_3''') \oplus (H'''';E_1'''',E_2'''',E_3'''')$.
\item $ E_1'''' \vee E_2'''' = H''''$
\item $E_1'''= E_2'''=0 \subset E_3'''$ \ \ \text {(\lq\lq distributive component ")} 
\end{enumerate}
Therefore 
\begin{align*}
\mathcal S & = (H;E_1,E_2,E_3)\\ 
& \cong (H';E_1',E_2',E_3') \oplus (H''';E_1''',E_2''',E_3''') \oplus (H'''';E_1'''',E_2'''',E_3'''')
\end{align*}
Moreover  $ E_1'''' \vee E_2'''' = H''''$ and $ E_1'''' \cap E_2'''' = 0$. 
We shall split out a distributive part and a double triangle part 
step by step using Lemma \ref{lemma:cap-zero}, Lemma \ref{lemma:vee}, 
Lemma \ref{lemma:closedness} and  Lemma \ref{lemma:decomposed}. 
Since closedness 
property is preserved after we split out one step, we can 
proceed  the next step.

Finally we can split out a double 
triangle part $(Q;Q_1,Q_2,Q_3)$,  because  in the final step 
we have that   
$$
Q_i \vee Q_j = Q, \ \ \ \text{ and } \ \ \ Q_i \wedge Q_j = 0,  
\ \ (i \not= j , i, j = 1,2,3). 
$$
and the rest part consists of finite direct sum of distributive systems 
of three subspaces. Hence we have (2). 
\end{proof}

\section{dense decomposition }

In an infinite-dimensional Hilbert space $H$, the algebraic linear sum 
$H' + H''$ of closed subspaces $H'$ and $H''$ is not necessarily 
closed. Therefore we cannot expect direct sum decomposition in general. 

\bigskip
\noindent
{\bf Example 8.}
Let $\mathcal S = (H;E_1,E_2,E_3)$ be a system 
of three subspaces in a Hilbert space $H$. Suppose that 
$E_1 \cap E_2 = 0$ and $E_1 + E_2$ is not closed and 
$E_1 \vee E_2 = \overline{E_1 + E_2} = E_3 = H$.  
Put $H' = E_1$ and $H'' = E_2$. Then $H$ has a 
"dense decomposition" $H = \overline{H' + H''}$ such that 
$H' \cap H'' = 0$. $E_1 = E_1 + 0$, $E_2 = 0 + E_2$ and 
$E_3 =  \overline{E_1 + E_2}$
 
\bigskip
\noindent
{\bf Definition.} Let $\mathcal S = (H;E_1,E_2,E_3)$ be a system 
of three subspaces in a Hilbert space $H$ and 
let $\mathcal S^{(k)} = (H^{(k)};E_1^{(k)},E_2^{(k)},E_3^{(k)})$ be systems 
of three subspaces in a Hilbert space $H^{(k)}$ for $k = 1,2,\dots,m$. Then 
$\mathcal S$ is said to have  a dense decomposition 
\[
\mathcal S \supset {\mathcal S}^{(1)} \oplus_{alg} 
\dots \oplus_{alg} {\mathcal S}^{(m)} \ \ \ (dense)
\]
if $H = \overline{H^{(1)} \oplus_{alg} \dots \oplus_{alg} H^{(m)}}$, 
$E_1 = \overline{E_1^{(1)} \oplus_{alg} \dots \oplus_{alg} E_1^{(m)}}$, \\
$E_2 = \overline{E_2^{(1)} \oplus_{alg} \dots \oplus_{alg} E_2^{(m)}}$ and 
$E_3 = \overline{E_3^{(1)} \oplus_{alg} \dots \oplus_{alg} E_3^{(m)}}$. 
In particular, $H^{(1)}, \dots.H^{(m)}$ are linearly independent.

In geneal, we define a dense decomposition of a system of 
$n$-subspaces similarly. 

For example,any system of two subspaces 
$\mathcal S = (H;E_1,E_2)$has a dense decomposition satisfying distributive law. In fact, 
$$
(E_1\cap E_2) \oplus_{alg} (K \oplus O) \oplus_{alg} 
\Im\left(
\begin{array}
{@{\,}cccc@{\,}}c^{2}&cs\\ cs&s^{2}
\end{array}
\right)
\oplus_{alg} (E_1\cap E_2^{\perp})
\oplus_{alg} (E_1^{\perp} \cap E_2) 
\oplus_{alg} (E_1^{\perp} \cap E_2^{\perp}).
$$
is a dense decomposition of $H$.
We expect that a certain class of  systems
$\mathcal S$ of three subspaces has a 
dense decomposition with a distributive part $H^{dis}$, a double triangle part 
$Q$ and six kinds of  pentagon parts 
$$
H^{\sigma} = \overline{\sum_{\sigma \in S_3} 
(E^{\sigma}_{\sigma(1)} \oplus_{alg} E^{\sigma}_{\sigma(3)})}
$$
(with $E^{\sigma}_{\sigma (3)} 
\supset E^{\sigma}_{\sigma (2)}$), 
for a permutation $\sigma \in S_3$ on three letters $\{1,2,3\}$. 

%\[
%\mathcal S \supset {\mathcal S}^{dis} \oplus_{alg} {\mathcal S}^{dt}
%\oplus_{alg} {\mathcal S}^{p} \ \ \ (dense), 
%\]
A distributive part is an algebraic sum of $2^3 = 8$ components
$$
H^{dis} = S \oplus_{alg} N_1 \oplus_{alg} 
N_2 \oplus_{alg} N_3 \oplus_{alg} 
      M_1 \oplus_{alg} M_2 \oplus_{alg} M_3 \oplus_{alg} L 
$$
and a double triangle part is a Hilbert space $Q$ with 
\[
Q = \overline{Q_1 \oplus_{alg} Q_2}
 = \overline{Q_2 \oplus_{alg} Q_3} = \overline{Q_3 \oplus_{alg} Q_1}.
\]
Then $(H^{\sigma};E^{\sigma}_1,E^{\sigma}_2,E^{\sigma}_3)$ 
form pentagons (with $E^{\sigma}_{\sigma (3)} 
\supset E^{\sigma}_{\sigma (2)}$), 
so that 
$$
E^{\sigma}_{\sigma (1)}\vee E^{\sigma}_{\sigma (2)} = H^{\sigma}, \ \  
E^{\sigma}_{\sigma (1)}
\wedge E^{\sigma}_{\sigma (3)} = 0,  
\text{ and } E^{\sigma}_{\sigma (3)} \supset E^{\sigma}_{\sigma (2)}
\text{ with  } E^{\sigma}_{\sigma (3)} \not= E^{\sigma}_{\sigma (2)}. 
$$
In this way 
$$
S \oplus_{alg} N_1 \oplus_{alg} 
N_2 \oplus_{alg} N_3 \oplus_{alg} 
      M_1 \oplus_{alg} M_2 \oplus_{alg} M_3 \oplus_{alg} L 
\oplus_{alg} Q  \oplus_{alg}
\sum_{\sigma \in S_3} 
(E^{\sigma}_{\sigma(1)} \oplus_{alg} E^{\sigma}_{\sigma(3)})
$$
is dense in $H$.

But we do not know whether this kinds of decompostion 
 hold or not in general.

Finally we give a partial result which gives a condition 
on a system to have a (dense) decomposition containing a pentagon. 

\begin{lemma}
Let $\mathcal S = (H;E_1,E_2,E_3)$ be a system 
of three subspaces in a Hilbert space $H$. Suppose that 
there exist systems $\mathcal S ' = (H';E_1',E_2',E_3')$ and 
$\mathcal S '' = (H'';E_1'',E_2'',E_3'')$ of three subspaces such that 
\begin{enumerate}
\item $\mathcal S = (H;E_1,E_2,E_3) 
\cong (H';E_1',E_2',E_3') \oplus (H'';E_1'',E_2'',E_3'')$.
\item $(H';E_1',E_2',E_3')$  forms a pentagon  (with $E_3' \supset E_2'$). 
\item (distributive component) there exist subspaces $N_1, N_2$  and 
$ M_1$ of $H''$ such that 
$$
H'' = N_1 \oplus N_2 \oplus M_1, 
$$ 
$E_1'' = 0 \oplus N_2 \oplus M_1$, $E_2'' = N_1 \oplus 0 \oplus 0$ and 
$E_3'' = N_1 \oplus N_2 \oplus 0$.  
\end{enumerate}
Then $E_1 \wedge E_2 = 0$, $E_1 \vee E_2 = H$ and $E_2 \varsubsetneqq E_3$. 
\end{lemma}
\begin{proof} 
It is clear. 
\end{proof}

We can rearrange the above decomposition such that 
$(H';E_1' \oplus M_1,E_2'\oplus N_1,E_3'\oplus N_1)$  
forms a pentagon  (with $E_3'\oplus N_1 \supset E_2'\oplus N_1$) and 
$(N_2;N_2,0,N_2)$ is a distributive part.

We shall show that the converse of the Lemma above holds in the 
sense of dense decomposition. 

\begin{prop}
Let $\mathcal S = (H;E_1,E_2,E_3)$ be a system 
of three subspaces in a Hilbert space $H$. Suppose that 
$$
E_1 \wedge E_2 = 0, E_1 \vee E_2 = H \text{ and }E_2 \varsubsetneqq E_3. 
$$
We assume that $E_3/E_2$ is finite dimensional. 
Then we have the following:\\

\noindent
{\rm (i)}If $E_3 \not= E_3 \cap (E_1 + E_2)$, 
then there exist subspaces $E_1',E_2',E_3'$ and $N_2$ of $H$  
such that 
\begin{enumerate}
\item  $H \supset E_1' \oplus_{alg} E_3' \oplus_{alg} N_2$   (dense) 
\item $(E_1' \vee E_3';E_1',E_2',E_3')$  forms a pentagon 
 (with $E_3' \supset E_2'$). 
\item  
$E_1 = \overline{E_1' \oplus_{alg}  N_2 }$, 
$E_2 = E_2'$ and 
$E_3 = \overline{E_3' \oplus_{alg}  N_2 }$ 
\end{enumerate}
Moreover $(N_2;N_2,0,N_2)$ is a distributive part. \\
\noindent
{\rm (ii)}If $E_3 = E_3 \cap (E_1 + E_2)$,  
then  there exist subspaces $N_1, N_2$  and 
$ M_1$ of $H$ such that 
$$
H\supset N_1 \oplus_{alg} N_2 \oplus_{alg} M_1 \text{ (dense) }, 
$$
$E_1 = 0 + N_2 + M_1$, $E_2 = N_1 + 0 + 0$ and 
$E_3 = N_1 + N_2 + 0$, 
\end{prop}
\begin{proof} 
Case {\rm (i)}: Assume that $E_3 \not= E_3 \cap (E_1 + E_2)$.

Put $F_3 = E_3 \cap (E_3 \cap (E_1 + E_2))^{\perp} \not= 0$. Then 
$E_3 = ( E_3 \cap (E_1 + E_2)) \oplus F_3$.  
Since $E_3 \supset E_2$, $F_3$ is orthgonal to $E_2$. We shall show that 
$$
E_1 \cap (E_2 + F_3) = 0
$$
In fact, let $x_1 = x_2 + f_3 \in E_1 \cap (E_2 + F_3)$ for 
$x_1 \in E_1$, $x_2 \in E_2$ and $f_3 \in F_3$.  Then 
$f_3 = x_1-x_2 \in E_3 \cap (E_1 + E_2)$. But $f_3$ is also in  
$(E_3 \cap (E_1 + E_2))^{\perp}$. Hence $f_3 =0$. Then $x_1 = x_2$ 
is in $E_1 \cap E_2 = 0$. Therefore $x_1 = x_2 = 0$.

Since $E_3/E_2$ is finite dimensional, we can find 
$$
u_1, \dots, u_n 
\in E_3 \cap (E_1 + E_2) (\supset E_2)
$$ 
such that the quotient image 
$\overline{u_1}, \dots, \overline{u_n}$ are linearly independent in 
the quotient $E_3/E_2$ and 
$$
E_3 \cap (E_1 + E_2) = E_2 + [u_1, \dots, u_n], 
$$
where $[u_1, \dots, u_n]$ is a linear span of $u_1, \dots, u_n$. 
In particular, $E_3 \cap (E_1 + E_2)$ is closed in $H$.  
Choose $v_1, \dots, v_n \in E_1$ and   $w_1, \dots, w_n \in E_2$ 
such that 
$$
u_k = v_k + w_k \ \ \ (k = 1, \dots, n).
$$
Since the quotient image $\overline{u_k} =\overline{v_k}$ for 
$k = 1, \dots, n$, we have that $v_1, \dots, v_n$ are 
linearly independent.

Put $N_2 = [v_1, \dots, v_n]$. Since $E_2 \subset E_3$, 
$v_k = u_k -w_k$ is also in $E_3$. Hence $N_2 \subset E_1 \cap E_3$. 
Since  $w_1, \dots, w_n$ are in $E_2$ , we have that 
$$
E_3 \cap (E_1 + E_2) = E_2 + [v_1, \dots, v_n], 
$$
Put $E_1' = E_1 \cap [v_1, \dots, v_n]^{\perp}$, $E_2' = E_2$ and 
$E_3' = E_2 \oplus F_3$. Then 
$$
E_3 = (E_3 \cap (E_1 + E_2)) \oplus F_3 = (E_2 + [v_1, \dots, v_n]) \oplus F_3.  
$$
We shall show that $N_2, E_1', E_2$ and $F_3$ are linearly independent. 
In fact, let $n_2 + x_1 + x_2 + f_3 = 0$  for 
$n_2 \in N_2, x_1 \in E_1', x_2 \in E_2$ and $f_3 \in F_3$.  Then 
$$
n_2 + x_1 = -x_2  -f_3 \in E_1 \cap (E_2 + F_3) = 0. 
$$
Therefore $n_2 + x_1 = 0$ and $x_2  + f_3 =0$. Since 
$N_2 \perp E_1'$ and $F_3 \perp E_2$, we have that 
$n_2 =x_1 = x_2  = f_3 = 0$.  
Since 
$$
E_1' \cap E_3' = E_1' \cap (E_2' + F_3)\subset E_1 \cap (E_2 + F_3) = 0, 
$$
$(E_1'\vee E_3' ;E_1',E_2',E_3')$  
forms a pentagon  (with $E_3' \supset E_2'$). 
And  
$E_1 = \overline{E_1' \oplus_{alg}  N_2 }$, 
$E_2 = E_2'$ and 
$E_3 = \overline{E_3' \oplus_{alg}  N_2 }$. \\
\noindent
Case {\rm (ii)}: Assume that $E_3 = E_3 \cap (E_1 + E_2)$. 
Then $F_3 = E_3 \cap (E_3 \cap (E_1 + E_2))^{\perp} = 0$. 
We can similarly define $u_k,v_k,w_k$ as above. 
Put $N_2 = [v_1, \dots, v_n]$.
Define $M_1 = E_1 \cap [v_1, \dots, v_n]^{\perp}$, $N_1 = E_2$. 
Then $E_1 = N_2 \oplus M_1$ and 
$E_3  = N_1 + N_2$. Moreover $N_1 + N_2 + M_1 = E_2 + E_1$ is 
dense in $H$. And $N_1, N_2$ and $M_1$ are linearly independent. In fact, 
let $n_1 + n_2 + m_1 = 0$  for $n_1 \in N_1$,$n_2 \in N_2$ and 
$m_1 \in M_1$. Then $n_2 + m_1 = -n_1 \in E_1 \cap E_2 = 0$. 
Hence $n_2 + m_1 = n_1 = 0$. Since $N_2$ and $M_1$ are orthogonal, 
$n_2 = m_1 = 0$.  Hence 
$$
H\supset N_1 \oplus_{alg} N_2 \oplus_{alg} M_1 \text{ (dense) }, 
$$ 
\end{proof}

\bigskip
\noindent{\bf Example 9}.Let $K = \ell^2({\mathbb N})$ be the 
Hilbert space of square summable sequences. Let 
$A: K \rightarrow K$  be a diagonal operator such that 
$(Ax)_n = \frac{1}{n} x_n$ for $x=(x_n)_n \in K$. 
Then $\Im A$ is dense in $K$ and not equal to $K$. Put 
$f = (1,1/2,1/3,\dots, 1/n, \dots  ) \in K$ and 
$v = (0,1/2,1/3,\dots, 1/n, \dots ) \in K$. 
Then $f$ and $v$ are not in $\Im A$. Define 
$N_2 = {\mathbb C}(0,v)$, 
 $H = K \oplus K$, $E_1 = (K \oplus 0) + N_2$, 
$E_2 = \{(x, Ax) | x \in K\}$ and  
$E_3 = E_2 + {\mathbb C}(0,f) + N_2$. Put
$E_1' = K \oplus 0$, $E_2' = \{(x, Ax) | x \in K\}$ and 
$E_3' = E_2 + {\mathbb C}(0,f)$.Then 
$(E_1' \vee E_3';E_1',E_2',E_3')$  forms a pentagon 
 (with $E_3' \supset E_2'$).

\end{document}